\documentclass
{amsart}
\usepackage{graphicx}
\newtheorem{thm}{Theorem}[section]
\theoremstyle{definition}
\newtheorem{cor}[thm]{Corollary}
\newtheorem{prop}[thm]{Proposition}
\newtheorem{defn}[thm]{Definition}
\newtheorem{lem}[thm]{Lemma}

\newtheorem{rem}[thm]{Remark}
\newtheorem{ex}[thm]{Example}

\newtheorem{co}[thm]{Question}
\numberwithin{equation}{section}
\begin{document}
\title[Classical and strongly classical 2-absorbing second submodules]{Classical and strongly classical 2-absorbing second submodules\footnote {This research was in part supported by a grant from IPM (No. 94130048)}}

\author[H. Ansari-Toroghy]{H. Ansari-Toroghy*}
\address{\llap{*\,}Department of pure Mathematics\\
Faculty of mathematical
Sciences\\
University of Guilan\\
P. O. Box 41335-19141, Rasht, Iran}
\email{ansari@guilan.ac.ir}

\author[F. Farshadifar]{F. Farshadifar**}
\address{\llap{**\,} (Corresponding Author), Department of Mathematics, Farhangian University, Tehran, Iran.}
\address{\llap{**\,} School of Mathematics, Institute for Research in Fundamental Sciences (IPM), P.O. Box: 19395-5746, Tehran, Iran}
\email{f.farshadifar@cfu.ac.ir}


\begin{abstract}
In this paper, we will introduce the concept of classical (resp. strongly classical) 2-absorbing second submodules of modules over a commutative ring as a generalization of 2-absorbing (resp. strongly 2-absorbing) second submodules and investigate some basic properties of these classes of modules.
\end{abstract}

\maketitle
\section{Introduction}
\noindent
Throughout this paper, $R$ will denote a commutative ring with
identity and ``$\subset$" will denote the strict inclusion. Further, $\Bbb Z$ will denote the ring of integers.

Let $M$ be an $R$-module. A proper submodule $P$ of $M$ is said
 to be \emph{prime} if for any $r \in R$ and $m \in M$ with
$rm \in P$, we have $m \in P$ or $r \in (P:_RM)$ \cite{Da78}. Let $N$ be
a submodule of $M$.
A non-zero submodule $S$ of $M$ is said to be \emph{second}
if for each $a \in R$, the homomorphism $ S \stackrel {a} \rightarrow S$
is either surjective or zero \cite{Y01}.
In this case $Ann_R(S)$ is a prime ideal of $R$.

The notion of 2-absorbing ideals as a generalization of prime ideals was introduced and studied in \cite{Ba07}. A proper ideal $I$ of $R$ is a \emph{2-absorbing ideal}
of $R$ if whenever $a, b, c \in R$ and $abc \in I$, then $ab \in I$ or
 $ac \in I$ or $bc \in I$. 
 The authors in \cite{YS11} and \cite{pb12}, extended 2-absorbing ideals
 to 2-absorbing submodules. A proper submodule $N$ of $M$ is called a \emph{2-absorbing submodule }of  $M$ if whenever $abm \in N$
for some $a, b \in R$ and $m \in M$, then $am \in N$ or $bm \in N$ or
$ab \in (N :_R M)$.

In \cite{AF16}, the present authors introduced the dual notion of 2-absorbing submodules (that is, \emph{2-absorbing (resp. strongly 2-absorbing) second submodules}) of $M$ and investigated some properties of these classes of modules. A non-zero submodule $N$ of $M$ is said to be a \emph{2-absorbing second submodule of} $M$ if whenever
 $a, b \in R$, $L$ is a completely irreducible submodule of $M$,
and $abN\subseteq L$, then $aN\subseteq L$ or $bN\subseteq L$ or $ab \in Ann_R(N)$.
A non-zero submodule $N$ of $M$ is said to be a \emph{strongly 2-absorbing second submodule of} $M$ if whenever
 $a, b \in R$, $K$ is a submodule of $M$,
and $abN\subseteq K$, then $aN\subseteq K$ or $bN\subseteq K$ or $ab \in Ann_R(N)$.

In \cite{mto16}, the authors introduced the notion of classical 2-absorbing
 submodules as a generalization of 2-absorbing submodules and studied some properties of this class of modules. A proper submodule $N$ of $M$ is called \emph{classical 2-absorbing submodule}
if whenever $a, b, c \in R$ and $m \in M$ with $abcm \in N$, then $abm \in N$ or $acm \in N$ or $bcm \in N$ \cite{mto16}.

The purpose of this paper is to introduce the concepts of classical and strongly classical 2-absorbing second submodules of an $R$-module $M$ as dual notion of classical 2-absorbing submodules and provide some information concerning these new classes of modules. We characterize classical (resp. strongly classical)  2-absorbing second submodules in Theorem \ref{t2.19} (resp. Theorem \ref{t52.19}). Also, we consider the relationship between classical 2-absorbing and strongly classical 2-absorbing second submodules in Examples \ref{e75.1}, \ref{e57.2}, and Propositions \ref{p3.5}. Theorem \ref{t2.24} (resp. Theorem \ref{t52.24}) of this paper shows that if $M$ is an Artinian $R$-module, then every non-zero submodule of $M$ has only a finite number of maximal
classical (resp. strongly classical) 2-absorbing second submodules. Further, among other results, we investigate strongly classical 2-absorbing second submodules of a
finite direct product of modules in Theorem \ref{t52.26}.

\section{Classical 2-absorbing second submodules}
\noindent
Let $M$ be an $R$-module. A proper submodule $N$ of
$M$ is said to be \emph{completely irreducible} if $N=\bigcap _
{i \in I}N_i$, where $ \{ N_i \}_{i \in I}$ is a family of
submodules of $M$, implies that $N=N_i$ for some $i \in I$. It is
easy to see that every submodule of $M$ is an intersection of
completely irreducible submodules of $M$ \cite{FHo06}.

We frequently use the following basic fact without further comment.
\begin{rem}\label{r2.1}
Let $N$ and $K$ be two submodules of an $R$-module $M$. To prove $N\subseteq K$, it is enough to show that if $L$ is a completely irreducible submodule of $M$ such that $K\subseteq L$, then $N\subseteq L$.
\end{rem}

\begin{defn}\label{d2.12}
Let $N$ be a non-zero submodule of an $R$-module $M$. We say that $N$ is a \emph{classical 2-absorbing second submodule of} $M$ if whenever $a, b, c \in R$, $L$ is a completely irreducible submodule of $M$, and $abcN\subseteq L$, then $abN\subseteq L$ or
$bcN\subseteq L$ or $acN \subseteq L$. We say $M$ is a \emph{classical 2-absorbing second module} if $M$ is a classical 2-absorbing second submodule of itself.
\end{defn}

\begin{thm}\label{t2.19}
Let $M$ be an $R$-module and $N$ be a non-zero submodule of $M$. Then the
following statements are equivalent:
\begin{itemize}
  \item [(a)] $N$ is a classical 2-absorbing second submodule of $M$;
  \item [(b)] For every $a, b \in R$ and  completely irreducible submodule $L$ of $M$ with $abN \not \subseteq L$, $(L:_RabN)=(L:_RaN) \cup(L:_RbN)$;
  \item [(c)] For every $a, b \in R$ and  completely irreducible submodule $L$ of $M$ with $abN \not \subseteq L$, $(L:_RabN)=(L:_RaN)$ or $(L:_RabN)=(L:_RbN)$;
  \item [(d)] For every $a, b \in R$, every ideal $I$ of $R$, and completely irreducible submodule $L$ of $M$ with $abIN \subseteq L$, either $abN \subseteq L$ or $aIN \subseteq L$ or $bIN \subseteq L$;
  \item [(e)] For every $a \in R$, every ideal $I$ of $R$, and completely irreducible submodule $L$ of $M$ with $aIN \not \subseteq L$, $(L:_RaIN)=(L:_RIN)$ or $(L:_RaIN)=(L:_RaN)$;
  \item [(f)] For every $a \in R$, ideals $I, J$ of $R$, and completely irreducible submodule $L$ of $M$ with $aIJN \subseteq L$, either $aIN \subseteq L$ or $aJN \subseteq L$ or $IJN \subseteq L$;
  \item [(g)] For ideals $I, J$ of $R$, and completely irreducible submodule $L$ of $M$ with $IJN \not \subseteq L$, $(L:_RIJN)=(L:_RIN)$ or $(L:_RIJN)=(L:_RJN)$;
  \item [(h)] For ideals $I_1, I_2, I_3$ of $R$, and completely irreducible submodule $L$ of $M$ with $I_1I_2I_3N \subseteq L$, either $I_1I_2N \subseteq L$ or $I_1I_3N \subseteq L$ or $I_2I_3N \subseteq L$;
  \item [(i)] For each completely irreducible submodule $L$ of $M$ with $N \not \subseteq L$, $(L:_RN)$ is a 2-absorbing ideal of $R$.
\end{itemize}
\end{thm}
\begin{proof}
$(a) \Rightarrow (b)$
Let $t \in (L:_RabN)$. Then $tabN \subseteq L$. Since $abN \not \subseteq L$, $atN \subseteq L$ or $btN \subseteq L$ as needed.

$(b) \Rightarrow (c)$
This follows from the fact that if an ideal is the union of two ideals, then it is equal to one of them.

$(c) \Rightarrow (d)$
Let for some $a, b \in R$, an ideal $I$ of $R$, and completely irreducible submodule $L$ of $M$, $abIN \subseteq L$. Then $I \subseteq  (L:_RabN)$. If $abN \subseteq L$, then we are done. Assume that $abN \not \subseteq L$. Then by part (c), $I \subseteq  (L:_RbN)$ or $I \subseteq  (L:_RaN)$  as desired.

$(d) \Rightarrow (e)\Rightarrow (f) \Rightarrow (g)\Rightarrow (h)$
The proofs are similar to that of the previous implications.

$(h) \Rightarrow (a)$
 Trivial.

$(h)\Leftrightarrow (i)$
This is straightforward.
\end{proof}

We recall that an $R$-module $M$ is said to be a \emph{cocyclic module} if
$Soc_R(M)$ is a large and simple submodule of $M$ \cite{Y98}. (Here $Soc_R(M)$
denotes the sum of all minimal submodules of $M$.) A submodule $L$ of $M$ is a completely irreducible submodule of $M$ if and only if $M/L$ is a cocyclic $R$-module \cite{FHo06}.
\begin{cor}\label{c2.20}
Let $N$ be a classical 2-absorbing second submodule of a cocyclic $R$-module $M$. Then $Ann_R(N)$ is a 2-absorbing ideal of $R$.
\end{cor}
\begin{proof}
This follows from Theorem \ref{t2.19} $(a) \Rightarrow (i)$, because $(0)$ is a completely irreducible submodule of $M$.
\end{proof}

\begin{ex}
For any prime integer $p$, let $M=\Bbb Z_{p^\infty}$ as a $\Bbb Z$-module and
$G_i=\langle 1/p^i+\Bbb Z\rangle$ for $i=1,2,3,\ldots$. Then $G_i$ is not a classical 2-absorbing second submodule of $M$ for $i=3,4,5,\ldots$.
\end{ex}

\begin{lem}\label{l9.1}
Every 2-absorbing second submodule of $M$ is a classical 2-absorbing second submodule of $M$.
\end{lem}
\begin{proof}
Let $N$ be a 2-absorbing second submodule of $M$, $a, b, c \in R$, $L$ a completely irreducible submodule of $M$, and $abcN\subseteq L$. Then $abN\subseteq (L:_Mc)$. Thus $aN \subseteq (L:_Mc)$  or $bN \subseteq (L:_Mc)$ or $abN=0$ because by \cite[2.1]{AFS124}, $(L:_Mc)$ is a completely irreducible submodule of $M$. Hence
$acN\subseteq L$ or $bcN\subseteq L$ or $abN\subseteq L$ as needed.
\end{proof}

\begin{ex}\label{e9.1}
Consider $M=\Bbb Z_{pq} \oplus \Bbb Q$ as a $\Bbb Z$-module, where $p, q$ are prime integers. Then $M$ is a classical 2-absorbing second module which is not a strongly 2-absorbing second module.
\end{ex}

\begin{prop}\label{c11.6}
Let $N$ be a classical 2-absorbing second submodule of an $R$-module $M$.
Then we have the following.
\begin{itemize}
   \item [(a)] If $a \in R$, then  $a^nN=a^{n+1}N$, for all $n \geq 2$.
   \item [(b)] If $L$ is a completely irreducible submodule of $M$ such that $N \not \subseteq L$, then $\sqrt{(L :_R N)}$ is a 2-absorbing ideal of $R$.
  \end{itemize}
\end{prop}
\begin{proof}
(a) It is enough to show that $a^2N=a^3N$.
It is clear that $a^3N \subseteq a^2N$. Let $L$ be a completely irreducible submodule of $M$ such that $a^3N \subseteq L$. Since $N$ is a classical 2-absorbing second submodule, $a^2N \subseteq L$. This implies that $a^2N \subseteq a^3N$.

(b) Assume that $a, b, c \in R$ and $abc \in \sqrt{(L :_R N)}$. Then there is a positive integer $t$ such that $a^tb^tc^tN \subseteq L$. By hypotheses, $N$ is a classical 2-absorbing second submodule of $M$, thus $a^tN \subseteq L$ or $b^tN \subseteq L$ or $c^tN \subseteq L$. Therefore, $a \in \sqrt{(L :_R N)}$ or $b \in \sqrt{(L :_R N)}$ or $c \in \sqrt{(L :_R N)}$.
\end{proof}

\begin{thm}\label{t1.1}
Let $N$ be a submodule of an $R$-module $M$. Then we have the following.
\begin{itemize}
\item [(a)] If $N$ is a classical 2-absorbing second submodule of
$M$, then $IN$ is a classical 2-absorbing second submodule of $M$ for all ideals $I$ of $R$ with $I \not \subseteq Ann_R(N)$.
\item [(b)] If $N$ is a classical 2-absorbing submodule of $M$, then $(N:_RI)$ is a classical 2-absorbing submodule of $M$ for all ideals $I$ of $R$ with $I \not \subseteq (N:_RM)$.
\item [(c)] Let $f : M \rightarrow \acute{M}$ be a monomorphism of R-modules. If $\acute{N}$ is a classical 2-absorbing second submodule of $f(M)$, then $f^{-1}(\acute{N})$ is a classical 2-absorbing second submodule of $M$.
\end{itemize}
\end{thm}
\begin{proof}
(a) Let $I$ be an ideal of $R$ with $I \not \subseteq Ann_R(N)$, $a, b, c \in R$, $L$ be a completely irreducible submodule of $M$, and $abcIN\subseteq L$. Then $acN \subseteq L$ or $cbIN \subseteq L$ or $abIN=0$
by Theorem \ref{t2.19}$(a) \Rightarrow (d)$. If $cbIN \subseteq L$ or $abIN=0$, then we are done.
If $acN \subseteq L$, then $acIN \subseteq acN$ implies that $acIN \subseteq L$, as needed.
Since  $I \not \subseteq Ann_R(N)$, we have $IN$ is a non-zero submodule of $M$.

(b) Use the technique of part (a) and apply \cite[Theorem 2]{mto16}.

(c) If $f^{-1}(\acute{N})=0$, then $f(M) \cap \acute{N}=ff^{-1}(\acute{N})=f(0)=0$. Thus $\acute{N}=0$, a contradiction. Therefore, $f^{-1}(\acute{N})\not=0$. Now let $a, b, c \in R$, $L$ be a completely irreducible submodule of $M$, and $abcf^{-1}(\acute{N})\subseteq L$. Then
  $$
  abc\acute{N}=abc(f(M) \cap \acute{N})=abcff^{-1}(\acute{N})\subseteq f(L).
  $$
By \cite[3.14]{AF16}, $f(L)$ is a completely irreducible submodule of $f(M)$. Thus as $\acute{N}$ is a classical 2-absorbing second submodule, $ab\acute{N} \subseteq f(L)$ or $bc\acute{N} \subseteq f(L)$ or $ac\acute{N} \subseteq f(L)$. Therefore, $abf^{-1}(\acute{N}) \subseteq f^{-1}f(L)=L$ or $bcf^{-1}(\acute{N}) \subseteq f^{-1}f(L)=L$ or  $acf^{-1}(\acute{N}) \subseteq f^{-1}f(L)=L$, as desired.
\end{proof}

An $R$-module $M$ is said to be a \emph{multiplication module} if for every submodule $N$ of $M$ there exists an ideal $I$ of $R$ such that $N=IM$ \cite{Ba81}.

An $R$-module $M$ is said to be a \emph{comultiplication module} if for every submodule $N$ of $M$ there exists an ideal $I$ of $R$ such that $N=(0:_MI)$, equivalently, for each submodule $N$ of $M$, we have $N=(0:_MAnn_R(N))$ \cite{AF07}.

\begin{cor}\label{c11.1}
Let $M$ be an $R$-module. Then we have the following.
\begin{itemize}
\item [(a)] If $M$ is a multiplication classical 2-absorbing second $R$-module, then every non-zero submodule of $M$ is a classical 2-absorbing second submodule of $M$.
\item [(b)] If $M$ is a comultiplication module and the zero submodule of $M$ is a classical 2-absorbing submodule, then every proper submodule of $M$ is a classical 2-absorbing submodule of $M$.
\end{itemize}
\end{cor}
\begin{proof}
This follows from parts (a) and (b) of Lemma \ref{t1.1}.
\end{proof}

\begin{prop}\label{p2.21}
 Let $M$ be an $R$-module and $\{K_i\}_{i \in I}$ be a chain of classical
2-absorbing second submodules of $M$. Then $\sum_{i \in I}K_i$ is a classical 2-absorbing second submodule of $M$.
\end{prop}
\begin{proof}
Let $a, b, c \in R$, $L$ be a completely irreducible submodule of $M$, and $abc\sum_{i \in I}K_i \subseteq L$. Assume that $ab\sum_{i \in I}K_i \not \subseteq L$ and $ac\sum_{i \in I}K_i \not \subseteq L$. Then there are $m,n \in I$ where $abK_n \not \subseteq L$ and $acK_m \not \subseteq L$. Hence, for every $K_n \subseteq K_s$ and every $K_m \subseteq K_d$ we have that $abK_s \not \subseteq L$ and $acK_d \not \subseteq L$.  Therefore, for each submodule $K_h$ such that $K_n \subseteq K_h$ and $K_m \subseteq K_h$, we have $bcK_h \subseteq L$.  Hence $bc\sum_{i \in I}K_i \subseteq L$, as needed.
\end{proof}

\begin{defn}\label{d2.22}
We say that a classical 2-absorbing second submodule $N$ of an $R$-module $M$
is a \emph {maximal classical 2-absorbing second submodule} of a submodule
$K$ of $M$, if $N \subseteq K$ and there does not exist a classical 2-absorbing second submodule $T$ of $M$ such that $N \subset T \subset K$.
\end{defn}

\begin{lem}\label{l2.23}
 Let $M$ be an $R$-module. Then every classical 2-absorbing second submodule of $M$ is contained in a maximal classical 2-absorbing second submodule of $M$.
\end{lem}
\begin{proof}
This is proved easily by using Zorn's Lemma and Proposition \ref{p2.21}.
\end{proof}

\begin{thm}\label{t2.24} Let $M$ be an Artinian $R$-module. Then
every non-zero submodule of $M$ has only a finite number of maximal
classical 2-absorbing second submodules.
\end{thm}
\begin{proof}
Suppose that there exists a non-zero submodule $N$ of $M$
such that it has an infinite number of maximal classical 2-absorbing second submodules. Let $S$ be a submodule of $M$ chosen minimal such that $S$ has an infinite number of maximal classical 2-absorbing second submodules because $M$ is an Artinian $R$-module. Then $S$ is not a classical 2-absorbing second submodule. Thus there exist $a, b, c \in R$ and a completely irreducible submodule $L$ of $M$
such that $abcS\subseteq L$ but $abS \not \subseteq L$, $acS \not \subseteq L$, and $bcS \not \subseteq L$. Let $V$ be a maximal classical 2-absorbing second submodule of $M$ contained in $S$. Then $abV \subseteq L$ or $acV \subseteq L$ or $bcV \subseteq L$. Thus $V \subseteq (L:_Mab)$  or $V \subseteq (L:_Mac)$ or $V \subseteq (L:_Mbc)$. Therefore, $V \subseteq (L:_Sab)$  or $V \subseteq (L:_Sac)$ or $V \subseteq (L:_Sbc)$. By the choice of $S$, the modules $(L:_Sab)$, $(L:_Sac)$, and $(L:_Sbc)$ have only finitely many maximal classical 2-absorbing second submodules.
Therefore, there is only a finite number of possibilities for the
module $S$, which is a contradiction.
\end{proof}

\section{Strongly classical 2-absorbing second submodules}
\begin{defn}\label{d52.12}
Let $N$ be a non-zero submodule of an $R$-module $M$. We say that $N$ is a \emph{strongly classical 2-absorbing second submodule of} $M$ if whenever $a, b, c \in R$, $L_1, L_2, L_3$ are  completely irreducible submodules of $M$, and $abcN\subseteq L_1 \cap L_2 \cap L_3$, then $abN\subseteq  L_1 \cap L_2 \cap L_3$ or $bcN\subseteq  L_1 \cap L_2 \cap L_3$ or $acN \subseteq  L_1 \cap L_2 \cap L_3$. We say $M$ is a \emph{strongly classical 2-absorbing second module} if $M$ is a strongly classical 2-absorbing second submodule of itself.
\end{defn}

Clearly every strongly classical 2-absorbing second submodule is a classical 2-absorbing second submodule.
\begin{co}
Let $M$ be an $R$-module. Is every classical 2-absorbing second submodule of $M$ a strongly classical 2-absorbing second submodule of $M$?
\end{co}

\begin{ex}
The $\Bbb Z$-module $\Bbb Z$ has no strongly classical 2-absorbing second submodule.
\end{ex}

\begin{thm}\label{t52.19}
Let $M$ be an $R$-module and $N$ be a non-zero submodule of $M$. Then the
following statements are equivalent:
\begin{itemize}
  \item [(a)] $N$ is strongly classical 2-absorbing second;
  \item [(b)] If $a, b, c \in R$, $K$ is a submodule of $M$, and $abcN\subseteq K$, then $abN\subseteq  K$ or $bcN\subseteq K$ or $acN \subseteq K$;
  \item [(c)] For every $a, b, c \in R$, $abcN=abN$ or $abcN=acN$ or $abcN=bcN$;
  \item [(d)] For every $a, b \in R$ and  submodule $K$ of $M$ with $abN \not \subseteq K$, $(K:_RabN)=(K:_RaN) \cup(K:_RbN)$;
  \item [(e)] For every $a, b \in R$ and  submodule $K$ of $M$ with $abN \not \subseteq K$, $(K:_RabN)=(K:_RaN)$ or $(K:_RabN)=(K:_RbN)$;
  \item [(f)] For every $a, b \in R$, every ideal $I$ of $R$, and submodule $K$ of $M$ with $abIN \subseteq K$, either $abN \subseteq K$ or $aIN \subseteq K$ or $bIN \subseteq K$;
  \item [(g)] For every $a \in R$, every ideal $I$ of $R$, and submodule $K$ of $M$ with $aIN \not \subseteq K$, $(K:_RaIN)=(K:_RIN)$ or $(K:_RaIN)=(K:_RaN)$;
  \item [(h)] For every $a \in R$, ideals $I, J$ of $R$, and submodule $K$ of $M$ with $aIJN \subseteq K$, either $aIN \subseteq K$ or $aJN \subseteq K$ or $IJN \subseteq K$;
  \item [(i)] For ideals $I, J$ of $R$, and submodule $K$ of $M$ with $IJN \not \subseteq K$, $(K:_RIJN)=(K:_RIN)$ or $(K:_RIJN)=(K:_RJN)$;
  \item [(j)] For ideals $I_1, I_2, I_3$ of $R$, and submodule $K$ of $M$ with $I_1I_2I_3N \subseteq K$, either $I_1I_2N \subseteq K$ or $I_1I_3N \subseteq K$ or $I_2I_3N \subseteq K$;
  \item [(k)] For each submodule $K$ of $M$ with $N \not \subseteq K$, $(K:_RN)$ is a 2-absorbing ideal of $R$.
\end{itemize}
\end{thm}
\begin{proof}
$(a) \Rightarrow (b)$
Let $a, b, c \in R$, $K$ is a submodule of $M$, and $abcN\subseteq K$. Assume on the contrary that $abN\not \subseteq  K$, $bcN\not \subseteq K$, and $acN \not \subseteq K$. Then there exist completely irreducible submodules $L_1, L_2, L_3$ of $M$ such that $K$ is a submodule of them but $abN\not \subseteq  L_1$, $bcN\not \subseteq L_2$, and $acN \not \subseteq L_3$. Now we have $abcN\subseteq L_1 \cap L_2 \cap L_3$. Thus by part (a), $abN\subseteq L_1 \cap L_2 \cap L_3$ or $bcN\subseteq L_1 \cap L_2 \cap L_3$ or $acN\subseteq L_1 \cap L_2 \cap L_3$. Therefore, $abN \subseteq L_1$ or $bcN \subseteq L_2$ or $acN \subseteq L_3$ which are contradictions.

$(b) \Rightarrow (c)$
Let $a, b, c \in R$. Then $abcN \subseteq abcN$ implies that $abN \subseteq abcN$ or $bcN \subseteq abcN$ or $acN \subseteq abcN$ by part (b). Thus $abN=abcN$ or $bcN=abcN$ or $acN=abcN$ because the reverse inclusions are clear..

$(c) \Rightarrow (d)$
Let $t \in (K:_RabN)$. Then $tabN \subseteq K$. Since $abN \not \subseteq K$, $atN \subseteq K$ or $btN \subseteq K$ as needed.

$(d) \Rightarrow (e)$
This follows from the fact that if an ideal is the union of two ideals, then it is equal to one of them.

$(e) \Rightarrow (f)$
Let for some $a, b \in R$, an ideal $I$ of $R$, and submodule $K$ of $M$, $abIN \subseteq K$. Then $I \subseteq  (K:_RabN)$. If $abN \subseteq K$, then we are done. Assume that $abN \not \subseteq K$. Then by part (d), $I \subseteq  (K:_RbN)$ or $I \subseteq  (K:_RaN)$  as desired.

$(g) \Rightarrow (h)\Rightarrow (i) \Rightarrow (h)\Rightarrow (j)$
Have proofs similar to that of the previous implications.

$(j) \Rightarrow (a)$
 Trivial.

$(j)\Leftrightarrow (k)$
This is straightforward.
\end{proof}

Let $N$ be a submodule of an $R$-module $M$. Then Theorem \ref{t52.19} $(a) \Leftrightarrow (c)$ shows that $N$ is a strongly classical 2-absorbing second submodule of $M$ if and only if $N$ is a  strongly classical 2-absorbing second module.

\begin{cor}\label{c1.6}
Let $N$ be a strongly classical 2-absorbing second submodule of an $R$-module $M$ and $I$ be an ideal of $R$. Then  $I^nN=I^{n+1}N$, for all $n \geq 2$.
\end{cor}
\begin{proof}
It is enough to show that $I^2N=I^3N$. By Theorem \ref{t52.19}, $I^2N=I^3N$.
\end{proof}

\begin{ex}\label{e59.1}
Clearly every strongly 2-absorbing second submodule is a strongly classical 2-absorbing second submodule. But the converse is not true in general. For example, consider $M=\Bbb Z_6 \oplus \Bbb Q$ as a $\Bbb Z$-module. Then $M$ is a strongly classical 2-absorbing second module. But $M$ is not a strongly 2-absorbing second module.
\end{ex}

A non-zero submodule $N$ of an $R$-module $M$ is said to be a \emph{weakly second submodule} of $M$ if $rsN\subseteq K$, where $r,s \in R$ and $K$ is a submodule of
$M$, implies either $rN\subseteq K$ or $sN\subseteq K$ \cite{AF101}.

\begin{prop}\label{p52.20}
Let $M$ be an $R$-module. Then we have the following.
 \begin{itemize}
   \item [(a)] If $M$ is a comultiplication $R$-module and $N$ is a strongly classical 2-absorbing second submodule of $M$, then $N$ is a strongly  2-absorbing second submodule of $M$.
   \item [(b)] If $N_1$, $N_2$ are weakly second submodules
               of $M$,  then $N_1+N_2$ is a strongly classical 2-absorbing second submodule of $M$.
   \item [(c)] If $N$ is a strongly classical 2-absorbing second submodule of
               $M$, then $IN$ is a strongly classical 2-absorbing second submodule of $M$ for all ideals $I$ of $R$ with $I \not \subseteq Ann_R(N)$.
   \item [(d)] If $M$ is a multiplication strongly classical 2-absorbing second $R$-module, then every non-zero submodule of $M$ is a classical 2-absorbing second submodule of $M$.
   \item [(e)] If $M$ is a strongly classical 2-absorbing second $R$-module, then every non-zero homomorphic image of $M$ is a classical 2-absorbing second $R$-module.
 \end{itemize}
 \end{prop}
\begin{proof}
(a) By Theorem \ref{t52.19} $(a) \Rightarrow (k)$, $Ann_R(N)$ is a 2-absorbing ideal of $R$. Now the result follows from \cite[3.10]{AF16}.

(b) Let $N_1$, $N_2$ be weakly second submodules of $M$ and  $a, b, c \in R $. Since $N_1$ is a weakly second submodule, we may assume that $abcN_1=aN_1$. Likewise, assume that $abcN_2=bN_2$. Hence $abc(N_1+N_2)=ab(N_1+N_2)$ which implies $N_1+N_2$ is a classical 2-absorbing second submodule by Theorem \ref{t52.19} $(c) \Rightarrow (a)$.

(c) Use the technique of the proof of Theorem \ref{t1.1} (a).

(d) This follows from part (c).

(e) This is straightforward.
\end{proof}

For a submodule $N$ of an $R$-module $M$ the the \emph{second radical} (or \emph{second socle}) of $N$ is defined  as the sum of all second submodules of $M$ contained in $N$ and it is denoted by $sec(N)$ (or $soc(N)$). In case $N$ does not contain any second submodule, the second radical of $N$ is defined to be $(0)$ (see \cite{CAS13} and \cite{AF11}).
\begin{thm}\label{p111.11}
Let $M$ be a finitely generated comultiplication $R$-module. If $N$ is a strongly classical 2-absorbing second submodule of $M$, then $sec(N)$ s a strongly 2-absorbing second submodule of $M$.
\end{thm}
\begin{proof}
Let $N$ be a strongly classical 2-absorbing second submodule of $M$. By Proposition \ref{p52.20} (a), $Ann_R(N)$ is a 2-absorbing ideal of $R$. Thus by \cite[2.1]{Ba07}, $\sqrt{Ann_R(N)}$ is a 2-absorbing ideal of $R$. By \cite[2.12]{AF25}, $Ann_R(sec(N))=\sqrt{Ann_R(N)}$. Therefore, $Ann_R(sec(N))$ is a 2-absorbing ideal of $R$. Now the result follows from \cite[3.10]{AF16}.
\end{proof}

The following examples show that the two concepts of classical 2-absorbing submodules and strongly classical 2-absorbing second submodules are different in general.
\begin{ex} \label{e75.1}
The submodule $2\Bbb Z$ of the $\Bbb Z$-module $\Bbb Z$ is a classical 2-absorbing submodule which is not a strongly classical 2-absorbing second module.
\end{ex}

\begin{ex} \label{e57.2}
The submodule $\langle 1/p+\Bbb Z\rangle$ of the $\Bbb Z$-module $\Bbb Z_{p^\infty}$ is a a strongly classical 2-absorbing second module which is not a classical 2-absorbing submodule of $\Bbb Z_{p^\infty}$.
\end{ex}

A commutative ring $R$ is said to be a \textit{$u$-ring} provided
$R$ has the property that an ideal contained in a finite union of ideals must be
contained in one of those ideals; and a \textit{$um$-ring} is a ring $R$ with the property that an $R$-module which is equal to a finite union of submodules must be equal to one of them \cite{QB75}.

In the following proposition, we investigate the relationships between strongly classical 2-absorbing second submodules and classical 2-absorbing submodules.
\begin{prop} \label{p3.5} Let $M$ be a non-zero $R$-module. Then we have the following.
\begin{itemize}
\item [(a)] If $M$ is a finitely generated strongly classical 2-absorbing second
$R$-module, then the zero submodule of $M$ is a classical 2-absorbing submodule.
\item [(b)] If $M$ is a multiplication strongly classical 2-absorbing second $R$-module, then the zero submodule of $M$ is a classical 2-absorbing submodule.
\item [(c)] Let R be a $um$-ring. If $M$ is a Artinian $R$-module and the zero submodule of $M$ is a classical 2-absorbing submodule, then $M$ is a strongly classical 2-absorbing second $R$-module.
\item [(d)] Let R be a $um$-ring. If $M$ is a comultiplication $R$-module and the zero submodule of $M$ is a classical 2-absorbing submodule, then $M$ is a strongly classical 2-absorbing second $R$-module.
\end{itemize}
\end{prop}
\begin{proof}
(a) Let $a, b ,c \in R$, $m \in M$, and $abcm =0$. By Theorem \ref{t52.19}, we can
 assume that $abcM=acM$. Since $M$ is finitely generated, by using \cite[Theorem 76]{Kap11}, $Ann_R(abM)+Rc=R$. It follows that $(0:_Mabc)=(0:_Mab)$. This implies that $abm =0$, as needed.

(b) Let $a, b ,c \in R$, $m \in M$, and $abcm =0$. Then by Theorem \ref{t52.19}, we can
 assume that $abcM=acM$.  Thus
$$
0=abc((0:_Mabc):_RM)M=(((0:_Mabc):_RM)M)ab.
$$
Since $M$ is a multiplication module, $((0:_Mabc):_RM)M=(0:_Mabc)$.
Therefore, $(0:_Mabc)ab=0$.
It follows that $(0:_Mabc)\subseteq (0:_Mab)$. Thus $(0:_Mabc)= (0:_Mab)$ because the reverse inclusion is clear. Hence $abm =0$, as required.

(c) Let $a, b, c \in R$. Then by \cite[Theorem 4]{mto16}, we can assume that $(0:_Mabc)=(0:_Mab)$. Hence $(0:_{M/(0:_Mab)}c)=0$. Since $M$ is Artinian, it follows that $cM+(0:_Mab)=M$. Therefore, $abcM=abM$. Thus by Theorem \ref{t52.19} $(c) \Rightarrow (a)$, $M$
is a classical 2-absorbing second $R$-module.

(d) Let $a, b, c \in R$. Then by \cite[Theorem 4]{mto16}, we can assume that $(0:_Mabc)=(0:_Mab)$.
Since $M$ is a comultiplication $R$-module, this implies that
$$
M=((0:_Mabc):_MAnn_R(abcM)=((0:_Mab):_MAnn_R(abcM))=(abcM:_Mab).
$$
It follows that $abM\subseteq abcM$. Thus $abM=abcM$ because the reverse implication is clear and this completed the proof.
\end{proof}

\begin{prop}\label{p52.21}
 Let $M$ be an $R$-module and $\{K_i\}_{i \in I}$ be a chain of strongly classical
2-absorbing second submodules of $M$. Then $\sum_{i \in I}K_i$ is a strongly classical 2-absorbing second submodule of $M$.
\end{prop}
\begin{proof}
Use the technique of Proposition \ref{p2.21}.
\end{proof}

\begin{defn}\label{d52.22}
We say that a strongly classical 2-absorbing second submodule $N$ of an $R$-module $M$
is a \emph {maximal strongly classical 2-absorbing second submodule} of a submodule
$K$ of $M$, if $N \subseteq K$ and there does not exist a strongly classical 2-absorbing second submodule $T$ of $M$ such that $N \subset T \subset K$.
\end{defn}

\begin{lem}\label{l52.23}
 Let $M$ be an $R$-module. Then every strongly classical 2-absorbing second submodule of $M$ is contained in a maximal strongly classical 2-absorbing second submodule of $M$.
\end{lem}
\begin{proof}
This is proved easily by using Zorn's Lemma and Proposition \ref{p52.21}.
\end{proof}

\begin{thm}\label{t52.24} Let $M$ be an Artinian $R$-module. Then
every non-zero submodule of $M$ has only a finite number of maximal
strongly classical 2-absorbing second submodules.
\end{thm}
\begin{proof}
Use the technique of Theorem \ref{t2.24} any apply Lemma \ref{l52.23}.
\end{proof}

\begin{thm}\label{t522.16}
Let $f : M \rightarrow \acute{M}$ be a monomorphism of R-modules. Then we have the following.
\begin{itemize}
  \item [(a)] If $N$ is a strongly classical 2-absorbing second submodule of $M$, then $f(N)$ is a strongly classical 2-absorbing second submodule of $\acute{M}$.
  \item [(b)] If $\acute{N}$ is a strongly classical 2-absorbing second submodule of $f(M)$, then $f^{-1}(\acute{N})$ is a strongly classical 2-absorbing second submodule of $M$.
\end{itemize}
\end{thm}
\begin{proof}
(a) Since $N \not =0$ and $f$ is a monomorphism, we have $f(N) \not =0$. Let $a, b, c \in R$. Then by Theorem \ref{t52.19} $(a)\Rightarrow (c)$, we can assume that $abcN=abN$. Thus
$$
abcf(N)=f(abcN)=f(abN)=abf(N).
$$
Hence $f(N)$ is a classical 2-absorbing second submodule of $\acute{M}$ by Theorem \ref{t52.19} $(c)\Rightarrow (a)$.

(b) If $f^{-1}(\acute{N})=0$, then $f(M) \cap \acute{N}=ff^{-1}(\acute{N})=f(0)=0$. Thus $\acute{N}=0$, a contradiction. Therefore, $f^{-1}(\acute{N})\not=0$. Now let $a, b, c \in R$, $K$ be a submodule of $M$, and $abcf^{-1}(\acute{N})\subseteq K$. Then
  $$
  abc\acute{N}=abc(f(M) \cap \acute{N})=abcff^{-1}(\acute{N})\subseteq f(K).
  $$
Thus as $\acute{N}$ is a strongly classical 2-absorbing second submodule, $ab\acute{N} \subseteq f(K)$ or $bc\acute{N} \subseteq f(K)$ or $ac\acute{N} \subseteq f(K)$. Therefore, $abf^{-1}(\acute{N}) \subseteq f^{-1}f(K)=K$ or $bcf^{-1}(\acute{N}) \subseteq f^{-1}f(K)=K$ or  $acf^{-1}(\acute{N}) \subseteq f^{-1}f(K)=K$, as desired.
\end{proof}

Let $R_i$ be a commutative ring with identity and $M_i$ be an $R_i$-module for $i = 1, 2$. Let $R = R_1 \times R_2$. Then $M = M_1 \times M_2$ is an $R$-module and each submodule of $M$ is in the form of $N = N_1 \times N_2$ for some submodules $N_1$ of $M_1$ and $N_2$ of $M_2$.
\begin{thm}\label{t52.224}
Let $R = R_1 \times R_2$ be a decomposable ring and let $M = M_1 \times M_2$
be an $R$-module, where $M_1$ is an $R_1$-module and $M_2$ is an $R_2$-module. Suppose that $N = N_1 \times N_2$ is a non-zero submodule of $M$. Then the following conditions are equivalent:
\begin{itemize}
  \item [(a)] $N$ is a strongly classical 2-absorbing second submodule of $M$;
  \item [(b)] Either $N_1 = 0$ and $N_2$ is a strongly classical 2-absorbing second submodule of $M_2$ or $N_2 = 0$ and $N_1$ is a strongly classical 2-absorbing second submodule of $M_1$ or $N_1$, $N_2$ are weakly second submodules of $M_1$, $M_2$, respectively.
\end{itemize}
\end{thm}
\begin{proof}
$(a) \Rightarrow (b)$.  Suppose that $N$ is a strongly classical 2-absorbing second
submodule of $M$ such that $N_2 = 0$. From our hypothesis, $N$ is non-zero, so $N_1 \not =0$. Set $\acute{M}=M_1\times 0$. One can see that $\acute{N}=N_1 \times 0$ is a strongly classical 2-absorbing second submodule of $\acute{M}$.
Also observe that $\acute{M} \cong M_1$ and $\acute{N} \cong N_1$. Thus $N_1$ is a strongly classical 2-absorbing second submodule of $M_1$. Suppose that $N_1\not =0$ and $N_2 \not =0$. We show that $N_1$ is a weakly second submodule of $M_1$. Since $N_2 \not =0$, there exists a completely irreducible submodule $L_2$ of $M_2$ such that $N_2 \not \subseteq L_2$. Let $abN_1 \subseteq K$ for some $a, b \in R_1$ and submodule $K$ of $M_1$. Thus $(a, 1)(b, 1)(1, 0)(N_1\times N_2) =
abN_1\times 0 \subseteq  K \times L_2$. So either $(a, 1)(b, 1)(N_1\times N_2) = abN_1\times N_2 \subseteq K \times L_2$ or  $(a, 1)(1, 0)(N_1\times N_2) = aN_1\times 0\subseteq K \times L_2$ or $(b, 1)(1, 0)(N1\times N_2) = bN_1\times 0 \subseteq K \times L_2$. If $abN_1\times N_2 \subseteq K \times L_2$, then $N_2 \subseteq L_2$, a contradiction. Hence either $aN_1\subseteq K$ or $bN_1\subseteq K$ which shows that $N_1$ is a weakly second submodule of $M_1$. Similarly, we can show that $N_2$ is a weakly second submodule of $M_2$.

$(b) \Rightarrow (a)$.
Suppose that $N = N_1 \times 0$, where $N_1$ is a strongly classical 2-absorbing (resp.
weakly) second submodule of $M_1$. Then it is clear that $N$ is a strongly classical 2-absorbing (resp. weakly) second submodule of $M$. Now, assume that $N = N_1 \times N_2$, where $N_1$ and $N_2$ are weakly second submodules of $M_1$ and $M_2$, respectively. Hence $(N_1 \times 0)+ (0 \times N_2) = N_1 \times N_2 = N$ is a strongly classical 2-absorbing second submodule of $M$, by Preposition \ref{p52.20} (b).
\end{proof}

\begin{lem}\label{l52.25}
Let $R = R_1\times R_2\times \cdots \times R_n$ be a decomposable ring and $M =
M_1 \times M_2 \cdots \times M_n$ be an $R$-module where for every $1\leq i \leq n$, $M_i$ is an $R_i$-module, respectively. A non-zero submodule $N$ of $M$ is a weakly second submodule of $M$ if and only if $N =\times^n_{i=1}N_i$ such that for some $k \in \{1, 2, ..., n\}$, $N_k$ is a weakly second submodule of $M_k$, and $N_i =0$ for every $i\in \{1, 2, ..., n\}\setminus \{k\}$.
\end{lem}
\begin{proof}
($\Rightarrow$)
Let $N$ be a weakly second submodule of $M$. We know $N =\times^n_{i=1}N_i$
where for every $1\leq i \leq n$, $N_i$ is a submodule of $M_i$, respectively. Assume that $N_r$ is a non-zero submodule of $M_r$ and $N_s$ is a non-zero submodule of $M_s$ for some $1\leq r <s \leq n$. Since $N$ is a weakly second submodule of $M$,
$$
(0,\cdots, 0, 1_{R_r},0, \cdots, 0)(0,\cdots, 0, 1_{R_s},0, \cdots, 0)N=(0,\cdots, 0, 1_{R_r},0, \cdots, 0)N
$$
or
$$
(0,\cdots, 0, 1_{R_r},0, \cdots, 0)(0,\cdots, 0, 1_{R_s},0, \cdots, 0)N=(0,\cdots, 0, 1_{R_s},0, \cdots, 0)N.
$$
Thus $N_r=0$ or $N_s=0$. This contradiction shows that exactly one of the $N_i$’s is non-zero, say $N_k$. Now, we show that $N_k$ is a weakly second submodule of $M_k$. Let $a, b \in R_k$. Since $N$ is a weakly second submodule of $M$,
$$
(0,\cdots, 0, a,0, \cdots, 0)(0,\cdots, 0, b,0, \cdots, 0)N=(0,\cdots, 0, a,0, \cdots, 0)N
$$
or
$$
(0,\cdots, 0, a,0, \cdots, 0)(0,\cdots, 0, b,0, \cdots, 0)N=(0,\cdots, 0, b,0, \cdots, 0)N
$$
Thus
$abN_k=aN_k$ or $abN_k=bN_k$ as needed.

($\Leftarrow$) This is clear.
\end{proof}

\begin{thm}\label{t52.26}
Let $R = R_1\times R_2\times \cdots \times R_n$ ($2\leq n < \infty$) be a decomposable ring and $M =
M_1 \times M_2 \cdots \times M_n$ be an $R$-module, where for every $1\leq i \leq n$, $M_i$ is an $R_i$-module, respectively. Then for a non-zero submodule $N$ of $M$ the following conditions are equivalent:
\begin{itemize}
  \item [(a)] $N$ is a strongly classical 2-absorbing second submodule of $M$;
  \item [(b)]  Either $N =\times^n_{i=1}N_i$ such that for some $k \in \{1, 2, ..., n\}$, $N_k$ is a strongly classical 2-absorbing second submodule of $M_k$, and $N_i =0$ for every $i\in \{1, 2, ..., n\}\setminus \{k\}$ or $N =\times^n_{i=1}N_i$ such that for some $k,m \in \{1, 2, ..., n\}$,$N_k$ is a weakly second submodule of $M_k$, $N_m$ is a weakly second submodule of $M_m$, and $N_i =0$ for every $i \in \{1, 2, ..., n\}\setminus \{k,m\}$.
\end{itemize}
\end{thm}
\begin{proof}
We use induction on $n$. For $n = 2$ the result holds by Theorem \ref{t52.224}. Now
let $3\leq n < \infty$ and suppose that the result is valid when $K = M_1\times \cdots \times M_{n-1}$. We show that the result holds when $M = K \times M_n$. By Theorem \ref{t52.224}, $N$ is a strongly classical 2-absorbing second submodule of $M$ if and only if either $N = L \times 0$ for some strongly classical 2-absorbing second submodule $L$ of $K$ or $N = 0 \times L_n$ for some strongly classical 2-absorbing second submodule $L_n$ of $M_n$ or $N = L \times L_n$ for some weakly second submodule $L$ of $K$ and some weakly second submodule $L_n$ of $M_n$. Note that by Lemma \ref{l52.25}, a non-zero submodule $L$ of $K$ is a weakly second submodule of $K$ if and only
if $L = \times^{n-1}_{i=1}N_i$ such that for some $k \in \{1, 2, ..., n - 1\}$, $N_k$ is a weakly second submodule of $M_k$ and $N_i =0$ for every $i\in \{1, 2, ..., n - 1\}\setminus \{k\}$. Hence the claim is proved.
\end{proof}

\begin{ex}\label{e58.15}
Let $R$ be a Noetherian ring and let $E =\oplus _{m\in Max(R)}E(R/m)$. Then
for each 2-absorbing ideal $P$ of $R$, $(0 :_E P)$ is a strongly classical 2-absorbing second submodule of $E$.
\end{ex}
\begin{proof}
By using \cite[p. 147]{SP72}, $Hom_R(R/P,E)\not =0$. Now since $(0 :_E P)\cong Hom_R(R/P,E)$,  $(0 :_E P)$ is a strongly 2-absorbing second submodule of $E$ by \cite[3.27]{AF16}. Now the result follows from Example \ref{e59.1}.
\end{proof}

\begin{thm}\label{l58.11} Let R be a $um$-ring and $M$ be an $R$-module. If $E$ is an injective $R$-module and $N$ is a classical 2-absorbing submodule of $M$ such that $Hom_R(M/N,E) \not =0$, then $Hom_R(M/N,E)$ is a strongly classical 2-absorbing second $R$-module.
\end{thm}
\begin{proof}
Let $a, b, c \in R$. Since $N$ is a classical 2-absorbing submodule
of $M$, we can assume that $(N:_Mabc)=(N:_Mab)$ by \cite[Theorem 4]{mto16}.
Since $E$ is an injective $R$-module, by replacing $M$ with $M/N$ in \cite[3.13 (a)]{AF101}, we have $Hom_R(M/(N:_Mr), E)=rHom_R(M/N,E)$ for each $r \in R$. Therefore,
$$
abcHom_R(M/N, E)=Hom_R(M/(N:_Mabc), E)=
$$
$$
Hom_R(M/(N:_Mab), E)=abHom_R(M/N,E),
$$
as needed
\end{proof}

\begin{thm}\label{t58.13}
Let $M$ be a strongly classical 2-absorbing second $R$-module and $F$ be a right exact  linear covariant functor over the category of $R$-modules. Then $F(M)$ is a strongly classical 2-absorbing second $R$-module if $F(M) \not =0$.
\end{thm}
\begin{proof}
 This follows from \cite[3.14]{AF101} and Theorem \ref{t52.19} $(a) \Rightarrow (c)$.
\end{proof}

\begin{cor}\label{c58.14}
Let $M$ be an $R$-module, $S$ be a multiplicative subset of $R$ and $N$ be a strongly classical 2-absorbing second submodule of $M$. Then $S^{-1}N$ is a strongly classical 2-absorbing second submodule of $S^{-1}M$ if $S^{-1}N\not =0$.
\end{cor}
\begin{proof}
This follows from Theorem \ref{t58.13}.
\end{proof}
\bibliographystyle{amsplain}

\end{document}